\documentclass[12pt]{article}
\usepackage[centertags]{amsmath}
\usepackage{amsfonts}
\usepackage{amssymb}
\usepackage{latexsym}
\usepackage{amsthm}
\usepackage{newlfont}
\usepackage{graphicx}
\RequirePackage{srcltx}

\newlength{\defbaselineskip}
\newcommand{\setlinespacing}[1]%
           {\setlength{\baselineskip}{#1 \defbaselineskip}}

\newcommand{\N}{{\mathbb{N}}}

\newcommand{\actaqed}{\hfill $\actabox$}

\newcommand{\T}{\mathbb{{T}}}

  \def\R{{\mathbb R}}
{\medskip\noindent \textit{Proof of #1. }}%
{\actaqed \medskip}

\def\A{{\mathcal A}}

\def \Tr{\mathcal T}

\def \V{\mathcal V}
\def\R{{\mathbb R}}
\def\Z{\mathbb Z}

\def \T{\mathbb T}
\def \<{\langle}
\def\>{\rangle}
\def \L{\Lambda}

\def \de{\delta}

\def \supp{\operatorname{supp}}

\def\Gm{\Gamma}
\def\bx{\mathbf x}
\def\by{\mathbf y}
\def\bk{\mathbf k}
\def\bm{\mathbf m}
\def\bs{\mathbf s}
\def\bu{\mathbf u}

\def\bt{\beta}

\def\Om{\Omega}
\def\kb{\bar k}

\def\RR{{\mathbb R}}

\def\sph{\mathbb{S}^{d-1}}
\def\HH{\mathcal{H}}
  \def\R{{\mathbb R}}
  \def\f{\frac}
\def\sub{\substack}

\newtheorem{Theorem}{Theorem}[section]
\newtheorem{Lemma}{Lemma}[section]
\newtheorem{Definition}{Definition}[section]
\newtheorem{Proposition}{Proposition}[section]
\newtheorem{Remark}{Remark}[section]
\newtheorem{Example}{Example}[section]
\newtheorem{Corollary}{Corollary}[section]
\numberwithin{equation}{section}

\newcommand{\be}{\begin{equation}}
\newcommand{\ee}{\end{equation}}

\begin{document}
\title{Remez-type inequalities for the hyperbolic cross polynomials}
\author{V. Temlyakov 
 and S. Tikhonov
 \thanks
{V. Temlyakov, University of South Carolina and Steklov Institute of Mathematics,\quad\qquad\qquad
S. Tikhonov, ICREA,
 Centre de Recerca Matem\`{a}tica,
 and UAB}
 } 

 \maketitle

\begin{abstract}
{In this paper we study the Remez-type  inequalities for trigonometric polynomials with harmonics from  hyperbolic crosses.
The interrelation between the  Remez and  Nikolskii inequalities for individual functions and its applications are discussed.
}

\end{abstract}

\section{Introduction}
In many questions in analysis one deals with a problem of finding the best possible way to estimate the global norm $\|f\|_{X(\Omega)}$ in terms of local norms $\|f\|_{X(\Omega\setminus B )}$.
 In some cases, this problem can be reduced to the problem for certain approximation methods, in particular, polynomials. An important result in this topic is the Remez inequality.

  For algebraic polynomials $P_n$,
the Remez inequality establishes a sharp upper bound for
$\|P_n\|_{L_\infty[-1,1]}$
 if the measure of the subset of $[- 1,1]$, where the modulus of the polynomial is at most $1$, is known \cite{re}.
 A sharp multidimensional inequality for algebraic polynomials was obtained by  Brudnyi and
Ganzburg in \cite{br-ga}.

In the case of trigonometric polynomials $
T_n(x)=\sum_{|k|\le n} c_ke^{ikx}$,\;
$c_k\in \mathbb{C},$
the Remez inequality reads as follows: for any
 Lebesgue measurable set $B\subset \mathbb{T}$ we have
\begin{equation}\label{rem-or}
\|T_n\|_{L_\infty([0,2\pi))}\le C(n,|B|)
\|T_n\|_{L_\infty([0,2\pi) \setminus B)}. 
\end{equation}
In \cite{er}, (\ref{rem-or}) was proved with $C(n,|B|)=\exp({4n|B|})$ for
$|B|< \pi/ 2$;
the history of the question can be found in, e.g., \cite[Ch. 5]{bor}, \cite[Sec. 3]{ga}, and \cite{lor-g}.
The constant can be sharpened as $C(n,|B|)=\exp({2n|B|})$, see \cite[Th. 3.1]{ga}.

In case when  the measure  $|B|$ is big, that is, when $\pi/2<|B|<2\pi$, one has
$$ C(n,|B|) =\Big(\frac{17}{2\pi-|B|}\Big)^{2n},$$ see \cite{ga, nazarov} and references therein.

Asymptotics of the sharp constant in the Remez inequality was recently obtained in \cite{nurs} and \cite{misha} for $|B|\to 0$ and $|B|\to 2\pi$,
respectively.

Multidimensional variants of Remez' inequality  for trigonometric polynomials
$$T_{\mathbf{n}}(x)=
\sum_{|k_1|\le {n_1}}
\cdots
\sum_{|k_d|\le {n_d}}
c_{\mathbf{k}} e^{i({\mathbf{k}},x)},
\quad
c_{\mathbf{k}}\in \mathbb{C}, \quad x\in \mathbb{T}^d, \quad d\ge 1,
$$
were obtained in \cite{nurs}:
$$
\|T_{\mathbf{n}}\|_{L_\infty(\mathbb{T}^d)}\le
\exp\Big( 2d \big( |B|\prod_{j=1}^d{ n_j}\big)^{1/d} \Big)
\,
\|T_{\mathbf{n}}\|_{L_\infty(\mathbb{T}^d \setminus B)}
$$
for
$$
 |B|<
\Big(\frac{\pi}2\Big)^d  \frac{\Big(\min\limits_{1\leq j\leq d}{n_j}\Big)^d}{\prod_{j=1}^dn_j}.
$$
This improves the previous results for
the case of $n_1=\cdots=n_d$
from \cite{prymak} and \cite{kroo}.

\vskip 0.2cm

It is worth mentioning that
Remez inequalities for exponential polynomials $$p(t)=\sum^n_{k=1}c_ke^{\lambda_kt}, \qquad c_k,\lambda_k\in\mathbb C,$$
are
 sometimes called the Tur\'{a}n inequality after Paul Tur\'{a}n \cite{turan} who studied related inequalities for algebraic complex-valued  polynomials.
In \cite{nazarov},  Nazarov  proved that
 for an interval $I\subset\mathbb R$ and  a measurable set $E\subset I$ of positive Lebesgue measure one has $$\sup_{t\in I}|p(t)|\leq e^{\mu(I)
 \max|{\textnormal{ Re}}\,\lambda_k|}\bigg(\frac {A\mu(I)}{\mu(E)}\bigg)^{n-1}\sup_{t\in E}|p(t)|.$$
 Here,
$A >0$ is an absolute constant, independent of
$n$.

Many different applications of Remez type inequalities include extension theorems (see, e.g., \cite{Brudny, yomdin}) and polynomial inequalities
(see, e.g., \cite{ga, prymak, kroo}). 
Moreover, Remez inequalities were used to obtain 
 the  uncertainty principle relations of the type $$\|f\|^2_{L^2(\mathbb R)}\leq Ae^{A\mu(E)\mu(\Sigma)}\Big(\int_{\mathbb R\backslash E}|f|^2+\int_{\mathbb R\backslash\Sigma}|\widehat f|^2\Big)$$ for any function $f\in L^2(\mathbb R)$ (see \cite{nazarov})
and  Logvinenko--Sereda type theorems (see \cite{kov, nazarov}).

In \cite{Naz1}, the authors used the Remez inequalities
 to derive sharp dimension--free estimates for the distribution of values of polynomials in convex subsets in $\mathbb{R}^n$, which allows to obtain
 interesting results
about the distribution of zeroes of random analytic functions. This topic is closely related to the known
Kannan-Lov\'{a}sz-Simonovits lemma.
In addition, the Remez inequality turns out to  be useful to deal with
the Rademacher Fourier series $$ f(\theta)=\sum_{k\in\Bbb{Z}} \xi_k a_k e^{2\pi i k\theta}, $$ where $\xi_k$ are independent Rademacher random variables taking the values of $\pm 1$ with probability $1/2$ and the coefficient sequence $\{a_k\}\in \ell^2$.  In particular, in \cite{Naz2}
the authors obtain $L^p$ bounds for the logarithm of a Rademacher Fourier series.

Remez inequality is closely related to the so-called Bernstein type inequalities \cite{fefferman, yo}, which have  many applications in differential equations, potential theory, and dynamical systems, see \cite{yo}.

\vskip 0.3cm

The main goal of this paper is to prove the Remez-type inequalities for the hyperbolic cross trigonometric polynomials. We also establish connections between the Remez-type inequalities and the Nikol'skii-type inequalities in a general setting. We use the following
definitions of these inequalities.

\begin{Definition}\label{D1.1} We say that $f$ satisfies the Remez-type inequality with parameters $p$, $b$, $R$ (in other words, $RI(p,b,R)$ holds) if for any measurable $B\subset \Om$ with measure $|B|\le b$
\be\label{1.1}
\|f\|_{L_p(\Om)} \le R\|f\|_{L_p(\Om\setminus B)}.
\ee
\end{Definition}

\begin{Definition}\label{D1.2} For $p>q$ we say that $f$ satisfies the Nikol'skii-type inequality with parameters $p$, $q$, $C$, $m$ (in other words, $NI(p,q,C,m)$ holds) if
\be\label{1.2}
\|f\|_p \le Cm^\bt \|f\|_q,\quad \|f\|_p :=\|f\|_{L_p(\Om)},\quad  \bt:=1/q-1/p.
\ee
\end{Definition}

In Section 2 we establish that the Remez-type inequalities and the Nikol'skii-type inequalities are closely connected. A typical result, that shows that $RI$ implies $NI$ is Proposition \ref{P2.1}, which gives for all $0<q<p\le\infty$
$$
RI(\infty,b,R) \Rightarrow NI(p,q,R^{q\bt},1/b).
$$
A typical result in the opposite direction, which shows that $NI$ implies $RI$ is Proposition \ref{P2.2} and Remark \ref{R2.1}: we have for $0<q<p\le\infty$
$$
NI(p,q,C,m) \Rightarrow RI(q,(C'm)^{-1},2^{\max(1,1/q)}).
$$
It is well known and easy to derive from the interpolation inequality
$$
\|f\|_q \le \|f\|_v^\theta\|f\|_p^{1-\theta}, \quad 0<v<q<p\le\infty,\quad \theta := (1/q-1/p)(1/v-1/p)^{-1}
$$
that for $0<v<q<p\le \infty$
$$
NI(p,q,C,m) \Rightarrow NI(q,v,C',m).
$$
This indicates that the Nikol'skii-type inequalities "propagate" from bigger values of $p$, $q$ to
smaller values of $q$, $v$. In Section 2 we note that a similar effect holds for the Remez-type inequalities. For instance, we prove that (see Lemma \ref{L2.2}) for $0<q<p<\infty$
$$
 RI(p,b,R) \Rightarrow RI(q,b,R^{p/q}).
 $$

 The main results of the paper are in Section 3, where we study the Remez-type inequalities for the hyperbolic cross trigonometric polynomials. The above discussion shows that the strongest $RI$ are for $p=\infty$. In Section 3 we prove that (see Theorem \ref{T3.1})
 the $RI(\infty,b(N),R(N))$ holds for all polynomials from $\Tr(N)$ (see the definition in Section 3) for $b(N) \asymp (N(\log N)^{d-1})^{-1}$ and $R(N)\asymp (\log N)^{d-1}$. We also prove that the extra factor $R(N)$ cannot be substantially improved. Namely, Proposition \ref{P3.1}
shows that even if we make a stronger assumption on $b(N) \asymp (N(\log N)^{A})^{-1}$ with arbitrarily large fixed $A$, we still cannot replace $R(N)\asymp (\log N)^{d-1}$ by
$R(N)\asymp (\log N)^{(d-1)(1-\delta)}$ with some $\delta>0$. This indicates that the Remez-type inequalities for $p=\infty$ for the hyperbolic cross polynomials differ from their univariate counterparts. It is not surprising, because it is known (see \cite{Tmon}) that the same phenomenon holds for the Bernstein and Nikol'skii inequalities. In Section 3 we establish that contrary to the case $p=\infty$ in the case
$p<\infty$ the $RI$ has the form similar to the univariate case (see Theorem \ref{T3.1p}). In particular, this implies that Theorem \ref{T3.1p} is sharp. The problem of sharpness of the $RI$ in the case $p=\infty$ is open. For instance, we do not know what is the best rate of decay
of $b(N)$, which guarantees the $RI(\infty,b(N),R(N))$ with the above $R(N)\asymp (\log N)^{d-1}$. Theorem \ref{T3.1} shows that it is sufficient to take $b(N) \asymp (N(\log N)^{d-1})^{-1}$. However, Theorem \ref{T3.2} shows that we can expect some improvements
on the rate of $b(N)$.
Here is other very interesting open problem.

{\bf Open problem.} What is the best rate of $\{b(N)\}$ to guarantee that $RI(\infty,b(N),C(d))$
holds for $\Tr(N)$?

This problem might be related to the discretization problem discussed in Subsections 2.4 and 3.4.

As usual,  $f\ll g$ for $f,g\ge 0$  means that $f\le C g$ with $C$
independent of essential quantities, and  $f\asymp g$ means that $f\ll g\ll f$.

\section{Some general inequalities}
In this section we show how the Remez-type inequality for an individual function $f$
can be used to derive the Nikol'skii-type inequalities for $f$. Also, we show how the Nikol'skii-type inequalities imply the Remez-type inequalities. These results show that the Remez-type and the Nikol'skii-type inequalities are closely related. In addition, we show that the discretization inequality (see below for the definition) implies the Remez-type inequalities.
\subsection{Remez-type inequalities}
  Suppose $f$ is a continuous on a compact $\Omega$ function. Let $\mu$ be a normalized measure on $\Om$. Assume that the following Remez-type inequality holds: for any measurable $B\subset \Om$ with measure $|B|\le b$
\be\label{2.1}
\|f\|_{L_\infty(\Om)} \le R\|f\|_{L_\infty(\Om\setminus B)}.
\ee
 We now show that inequality (\ref{2.1}) implies the Remez-type inequality for $f$ in the $L_p(\Om)$, $0<p<\infty$.
 \begin{Lemma}\label{L2.1} We have for $0<p<\infty$
\be\label{2.1+}
 RI(\infty,b,R) \Rightarrow RI(p,b/2,2^{1/p}R).
\ee
 \end{Lemma}
 \begin{proof}
 We prove   inequality (\ref{1.1}) for $B$, satisfying $|B|\le b/2$.
Take any set $B\subset \Om$ with $|B|\le b/2$ and estimate
\be\label{2.2}
\int_B |f|^pd\mu \le |B|\|f\|_\infty^p,\quad \|f\|_\infty := \|f\|_{L_\infty(\Om)}.
\ee
Define $B'\subset \Om\setminus B$, $|B'|=b/2$, to be such that for all $x\in B'$ we have
\be\label{2.3}
|f(x)| \ge \sup_{u\in (\Om\setminus B)\setminus B'} |f(u)|.
\ee
Denote $B'' := B\cup B'$. Then $|B''|\le b$. By (\ref{2.1}) and (\ref{2.3}) for all $x\in B'$
$$
|f(x)| \ge \sup_{u\in \Om\setminus B''} |f(u)| \ge R^{-1}\|f\|_\infty.
$$
Therefore,
\be\label{2.4}
\|f\|_\infty^p \le \frac{2}{b} \int_{B'} (R|f|)^pd\mu.
\ee
Inequalities (\ref{2.2}) and (\ref{2.4}) imply
$$
\int_{\Om} |f|^pd\mu = \int_{B} |f|^pd\mu +\int_{\Om\setminus B} |f|^pd\mu \le
\int_{B'} (R|f|)^pd\mu + \int_{\Om\setminus B} |f|^pd\mu
$$
\be\label{2.5}
\le (R^p+1)\int_{\Om\setminus B} |f|^pd\mu.
\ee
In other words
\be\label{2.6}
\|f\|_{L_p(\Om)} \le (R^p+1)^{1/p}\|f\|_{L_p(\Om\setminus B)} \le 2^{1/p}R \|f\|_{L_p(\Om\setminus B)}
\ee
for any $B$ with $|B|\le b/2$.
\end{proof}

 \begin{Remark}\label{R2.0}
Note that the reverse implication of relation  (\ref{2.1+}) is not valid. More precisely, there are no positive constants $C_1$ and $C_2$ such that
for $0<p<\infty$
$$
RI(p,b,R)
 \Rightarrow  RI(\infty,C_1 b,C_2 R).
$$
This follows immediately from Proposition \ref{P3.1}
 and Theorem \ref{T3.1p}
 below.
 \end{Remark}

 \begin{Lemma}\label{L2.2} We have for $0<q<p<\infty$
 $$
 RI(p,b,R) \Rightarrow RI(q,b,R^{p/q}).
 $$
 \end{Lemma}

 \begin{proof} Let $p<\infty$ and let $f$ satisfy the $RI(p,b,R)$: for any $B$, $|B|\le b$ we have
 \be\label{0.1}
 \int_{\Om}|f|^p\mu \le R^p\int_{\Om\setminus B} |f|^pd\mu.
 \ee
 Let $B^*$ be a set of measure $b$ such that for all $x\in B^*$ we have
$$
|f(x)| \ge \sup_{u\in \Om\setminus B^*}|f(u)| =: T.
$$
Then, by (\ref{0.1}) for $f\neq 0$ we have $T>0$. It is clear that for any $B$, $|B|=b$ we have
$$
\int_{\Om\setminus B} |f|^qd\mu \ge \int_{\Om\setminus B^*} |f|^qd\mu.
$$
We estimate from below $\int_{\Om\setminus B^*} |f|^qd\mu$. By (\ref{0.1}) we get
$$
R^p\int_{\Om\setminus B^*} |f|^pd\mu \ge \int_{\Om } |f|^pd\mu = \int_{\Om\setminus B^*} |f|^pd\mu + \int_{  B^*} |f|^pd\mu
$$
and
\be\label{0.2}
(R^p-1)\int_{\Om\setminus B^*} |f|^pd\mu \ge \int_{  B^*} |f|^pd\mu.
\ee
Using the inequalities $|f|/T \le 1$ on $\Om\setminus B^*$ and $|f|/T \ge 1$ on $  B^*$ we write
$$
(R^p-1)\int_{\Om\setminus B^*} (|f|/T)^qd\mu \ge (R^p-1)\int_{\Om\setminus B^*} (|f|/T)^pd\mu
$$
and continue by (\ref{0.2})
\be\label{0.3}
\ge \int_{ B^*} (|f|/T)^pd\mu \ge \int_{  B^*} (|f|/T)^qd\mu.
\ee
This implies
$$
R^p\int_{\Om\setminus B^*} |f|^qd\mu \ge \int_{\Om } |f|^qd\mu,
$$
which completes the proof.
\end{proof}

Note that also as  in Remark \ref{R2.0} the Remez inequality for $p<\infty$ keeps only "strong monotonicity" property with respect to parameters.
 \begin{Remark}\label{R2.2}
There are no positive constants $C_1$ and $C_2$ such that
for
$0<v<p<\infty$
 \begin{equation}\label{1-5}
RI(v,b,R) \Rightarrow  RI(p,C_1 b, C_2R^{v/p}).
 \end{equation}
  \end{Remark}

  First we prove that the inequality
  $$
  NI(p,q,C,m)
 \Rightarrow
NI(q,v,C',m),\qquad 0<v<q<p\le \infty,
$$
mentioned in the introduction,
is not invertible in the following sense.

 \begin{Remark}\label{R2.3}
The implication
$$
NI(q,v,C',m)
 \Rightarrow NI(p,q,C,c m), \qquad c>0,
$$
does not hold in general.
  \end{Remark}
Note that the case $p=\infty$ follows easily from Theorems \ref{T3.5} and \ref{T3.6} below.
In  the case $p<\infty$  we will use the following sharp Nikol'skii inequalities for spherical harmonics recently obtained in \cite{dd}.
Let 
$\HH_n^d$ be the space of all spherical harmonics of degree $n$ on   $\sph=\{ x\in \mathbb{R}^{d}:   \|x\|=1\}$,
 where $\|\cdot\|$ denotes the Euclidean norm of $\RR^d$.
 In particular,
 it is proved  in \cite{dd} that  for  $d\ge 3$ 

\begin{description} 
\item{({\em i})}
if  
 $1\leq  v\leq 2$ and $ v<q \leq
 \f{dv'}{d-2}$,
then
\begin{equation}\label{1-6}
\sup_{\sub{Y_n\in\mathcal{H}_n^d}}\f{\|Y_n\|_q}{\|Y_n\|_v} \asymp  n^{\f{d-2}2(\f 1v-\f 1q)},
\end{equation}

\item{({\em ii})} if $ \f{2d-2}{d-2}<q <p \leq \infty$, then
\begin{equation}\label{1-7}\sup_{\sub{Y_n\in\mathcal{H}_n^d}}\f{\|Y_n\|_p}{\|Y_n\|_q}\asymp  n^{(d-1) (\f 1q-\f1p)}.
\end{equation}

\end{description}
Since  $1\leq  v\leq 2$  always implies $ \f{2d-2}{d} <v'$, inequalities
(\ref{1-6}) and (\ref{1-7})
give
$$
NI(q,v,C_1(q,v),m)
 \nRightarrow NI(p,q,C_2(p,q),C_3 m)
$$
for
$$ \f{2d-2}{d-2}<q< \f{dv'}{d-2}$$
and
$$
1\leq  v\le 2<q<p\le \infty.
$$
This completes the proof of Remark \ref{R2.3}.

To prove Remark \ref{R2.2}, we first note that
 Proposition \ref{P2.2} implies that
$$
NI(q,v,C_1(q,v),m) \Rightarrow
RI(v,\f{C_1'(q,v)}{m},2^{\max({1,1/v})}),
\qquad
0< v<q\le \infty.
$$
On the other hand, Proposition  \ref{P2.1}
 yields  that
$$
RI(p,\f{C_1''(p,v,q)}{m},C_2(p,v)) \Rightarrow NI(p,q,C_2(p,v),\f{m}{C_1''(p,v,q)})
$$
for $0<q<p<\infty$.
Combining these estimates with
inequality (\ref{1-5}) for  $0<v<p<\infty$, we finally get
$$
NI(q,v,C_1(q,v),m) \Rightarrow
NI(p,q,C_2(p,v),\f{m}{C_1''(p,v,q)})
$$
for  $0<v<q<p<\infty$. These contradicts Remark \ref{R2.3}.
Thus,  the proof of Remark \ref{R2.2} is now complete.

\subsection{Remez-type inequality implies Nikol'skii-type inequalities}
First, we derive from (\ref{2.1}) Nikol'skii-type inequalities for $f$. We begin with estimating
$\|f\|_\infty$ in terms of $\|f\|_q$, $0<q<\infty$. Let, as above, $B^*$ be a set of measure $b$ such that for all $x\in B^*$ we have
$$
|f(x)| \ge \sup_{u\in \Om\setminus B^*}|f(u)|.
$$
Then by (\ref{2.1}) we have for all $x\in B^*$
$$
|f(x)| \ge R^{-1}\|f\|_\infty.
$$
Therefore,
$$
\|f\|_q^q \ge \int_{B^*}|f|^qd\mu \ge |B^*|(R^{-1}\|f\|_\infty)^q.
$$
We obtain from here
\be\label{2.7}
\|f\|_\infty \le Rb^{-1/q}\|f\|_q.
\ee
Let now $0<q<p<\infty$. We have
$$
\|f\|_p = \||f|^{1-q/p}|f|^{q/p}\|_p \le \|f\|_\infty^{1-q/p} \|f\|_q^{q/p}.
$$
Using (\ref{2.7}) we continue
$$
\le (Rb^{-1/q})^{1-q/p}\|f\|_q^{1-q/p}\|f\|_q^{q/p} = R^{q\bt}b^{-\bt}\|f\|_q,\quad \bt:=1/q-1/p.
$$
Thus we have proved the following statement.
\begin{Proposition}\label{P2.1} Remez-type inequality (\ref{2.1}) implies Nikol'skii-type
inequality
\be\label{2.8}
\|f\|_p \le R^{q\bt}b^{-\bt}\|f\|_q,\quad \bt:=1/q-1/p
\ee
for all $0<q<p\le\infty$. In other words,
$$
RI(\infty,b,R) \Rightarrow NI(p,q,R^{q\bt},1/b).
$$
\end{Proposition}

Second, we consider the case of $RI(p,b,R)$ with $p<\infty$.
\begin{Proposition}\label{P2.1p} For $0<q<p< \infty$ we have
$$
RI(p,b,R) \Rightarrow NI(p,q,R,1/b).
$$
\end{Proposition}
\begin{proof} We use the same notations as in the above proof of Lemma \ref{L2.2}.
First, we bound from above the thresholding parameter $T$. We have
$$
\|f\|_q^q \ge \int_{B^*}|f|^qd\mu \ge T^qb,
$$
which implies
\be\label{0.4}
T \le \|f\|_q b^{-1/q}.
\ee
Second, we estimate
\be\label{0.5}
\int_{\Om\setminus B^*} (|f|/T)^qd\mu \ge \int_{\Om\setminus B^*} (|f|/T)^pd\mu \ge T^{-p}R^{-p}\|f\|_p^p.
\ee
Relations (\ref{0.5}) and (\ref{0.4}) imply
\be\label{0.6}
\|f\|_p^p \le R^p T^{p-q} \|f\|_q^q \le R^pb^{-(p-q)/q}\|f\|_q^p
\ee
and
$$
\|f\|_p \le Rb^{-\bt}\|f\|_q.
$$
\end{proof}

Note that the statements of Propositions \ref{P2.1} and \ref{P2.1p}
are sharp in the following sense.
 \begin{Remark}\label{R3.1}
For $0<q<p\le \infty$
the implication
$$
NI(p,q,R,1/b)
\Rightarrow
RI(p,C_1 b,C_2 R({q,\bt})),\qquad C_1, C_2>0,
$$
does not hold in general.
  \end{Remark}

In particular, this follows from
Proposition \ref{P3.1} and
Theorem \ref{T3.5} taking  $p=\infty$ and $0<q \le 1$.
 \begin{Remark}\label{R3.2}
In light of Lemmas \ref{L2.1} and \ref{L2.2}, one can ask if
for $0<q<p\le \infty$
the following implication
$$
RI(q,b,R) \Rightarrow NI(p,q, C_1 R(p,q) ,C_2/b),\qquad C_1, C_2>0,
$$
holds, which is stronger than the one stated in  Propositions \ref{P2.1} and \ref{P2.1p}.
Again, Proposition \ref{P3.1} and
Theorem \ref{T3.5} with  $p=\infty$ and $0<q \le 1$
show that this is not the case.
  \end{Remark}

\subsection{Nikol'skii inequality implies Remez inequality}
We prove here the following statement.
\begin{Proposition}\label{P2.2} Suppose that a function $f$ satisfies the Nikol'skii inequality
\be\label{2.9}
\|f\|_p \le C(p,q)m^\bt \|f\|_q,\quad 1\le q<p\le \infty,\quad \bt:=1/q-1/p.
\ee
Then there exists a constant $C'(p,q)$ such that for any set $B\in \Om$, $|B|\le (C'(p,q)m)^{-1}$ we have
\be\label{2.9'}
\|f\|_{L_q(\Om)} \le 2\|f\|_{L_q(\Om\setminus B)}.
\ee
\end{Proposition}
\begin{proof} Denote $B^c := \Om\setminus B$ and $\chi_A$ the characteristic function of a set $A$. Then
\be\label{2.10}
\|f\|_q \le \|f\chi_{B^c}\|_q +\|f\chi_B\|_q.
\ee
Further, by H{\"o}lder inequality with parameter $p/q$ and our assumption (\ref{2.9}) we obtain
\be\label{2.11}
\|f\chi_B\|_q \le \|f\|_p |B|^\bt \le C(p,q)m^\bt |B|^\bt \|f\|_q.
\ee
Making measure $|B|$ small enough to satisfy $C(p,q)m^\bt |B|^\bt\le 1/2$ we derive from
 (\ref{2.11}) and (\ref{2.10}) the required inequality.
 \end{proof}
 \begin{Remark}\label{R2.1} Proposition \ref{P2.2} holds for all $0<q<p\le\infty$ with $2$ replaced by $2^{1/q}$ in (\ref{2.9'}) in case $q<1$.
 \end{Remark}

 \begin{Remark}\label{R2.2'}
 Remark \ref{R3.2} shows that the reverse statement to Proposition \ref{P2.2} does not hold in general.
 \end{Remark}

Propositions \ref{P2.1p} and \ref{P2.2} yield the following result.   

 \begin{Remark}\label{R2.3}
 Let $W_m$, $m\in \N$ or $m>0$, be a collection of subclasses of $L_r(\Omega)$, $A<r<B$.
 The following two conditions are equivalent:
\begin{description}
  \item{\textnormal{(i)}} for any $A<q<p<B$ we have
   $$
  \sup_{f\in W_m} \f{\|f\|_p}{\|f\|_q}\asymp  \lambda(m)^{ \f 1q-\f1p};
  $$

  \item{\textnormal{(ii)}} for any $A<r<B$ we have
   $$
  \sup_{f\in W_m} \sup \Big\{|B|:  \|f\|_{L_r(\Omega)} \le R\|f\|_{L_r(\Omega\setminus B)}
  \Big\}\asymp  \f1 {\lambda(m)}.
  $$
  \end{description}
 \end{Remark}
In many cases the Nikol'skii type inequalities are known.
  Proposition \ref{P2.2} and  Remark \ref{R2.1} allow us to derive the Remez type inequalities from these known results.
We illustrate this on some examples.

 \begin{Example}\label{E1.1}
\textnormal{(i).}
Taking into account the results from \cite{nessel},
for each trigonometric polynomial
$$T(x)=\sum_{k\in \supp \widehat{T}} c_k \exp(ikx), \qquad \supp \widehat{T}=\{k\in \Z^d: \widehat{T}(k)=c_k\ne 0 \}
$$
we have
$$
\|T\|_{L_p(\T^d)} \le 2^{\max(1,1/p)}\|T\|_{L_p(\T^d\setminus B)},
$$
where
$$|B|\le \f{1}{C}\left\{
           \begin{array}{ll}
             {1}/{N( \supp(\widehat{T}))}, & \hbox{$0<p< 2$;} \\
             {1}/{N( p_0 \textnormal{Conv}(\supp(\widehat{T})))}, & \hbox{$2\le p< \infty$,}
           \end{array}
         \right.
$$
$N(X)$ is the number of latice points in $X\subset\R^d$,
$p_0$ is the smallest integer not less than $p/2$, and $\textnormal{Conv}(\supp(\widehat{T}))$ denotes the convex hull of $\supp(\widehat{T}).$

\textnormal{(ii).}
 For each trigonometric polynomial
$$T_n(x)=\sum_{k=1}^n c_k \exp(i{n_k} x), \qquad n_k\in \Z, 
$$
we have
$$
\|T_n\|_{L_p(\T)} \le 2^{\max(1,1/p)}\|T_n\|_{L_p(\T\setminus B)},
$$
where
$$|B|\le \f{1}{C} \left\{
           \begin{array}{ll}
             {1}/{n }, & \hbox{$0<p\le 2$;} \\
             {1}/{n^{p/2}}, & \hbox{$2<p< \infty$.}
           \end{array}
         \right.
$$
This follows from results of Belinskii \cite{belin}.

\textnormal{(iii).} Sharp Nikol'skii inequalities for spherical harmonics  given by inequalities
(\ref{1-6}) and (\ref{1-7}) imply that
for any ${Y_n\in\mathcal{H}_n^d}$,
  $d\ge 3$,
we have that
$Y_n\in RI(p, 1/b, 2^{\max(1,1/p)})$ with
$$|B|\le \f{1}{C} \left\{
           \begin{array}{ll}
             {1}/{n^{\f{d-2}2} }, & \hbox{$0<p< 2$;} \\
             {1}/{n^{{d-1}} }, & \hbox{$\f{2d-2}{d-2}<p< \infty$.}
           \end{array}
         \right.
$$

\textnormal{(iv).}
Let $\Lambda_n=\{\lambda_0<\lambda_1\cdots<\lambda_n \}$ be a set of real numbers. Let us denote by $E(\Lambda_n)$ the collection of all linear combination of $e^{\lambda_0 t}, e^{\lambda_1 t}, \cdots, e^{\lambda_n t}$ over $\R$. Then the sharp Nikol'skii inequality from \cite{erd2007}
imply that

$$
\|P\|_{L_p([a,b])} \le 2^{\max(1,1/p)}\|P\|_{L_p([a,b]\setminus B)},\qquad P\in E(\Lambda_n),\qquad  0<p<\infty,
$$
where
$$|B|\le \f{1}{C\left(n^2+\sum_{j=1}^n|\lambda_j| \right)}.
$$

\textnormal{(v).} For functions $f$ such that $\supp \widehat{f}$ is  compact, using \cite{nessel},  we have  that

$$
\|f\|_{L_p(\R^d)} \le 2^{\max(1,1/p)}\|f\|_{L_p(\R^d\setminus B)},
$$
where
$$|B|\le \f{1}{C}\left\{
           \begin{array}{ll}
             {1}/{\mu ( \supp(\widehat{f}))}, & \hbox{$0<p< 2$;} \\
             {1}/{ p_0^n \mu (\textnormal{Conv}(\supp(\widehat{T})))}, & \hbox{$2\le p< \infty$,}
           \end{array}
         \right.
$$
$\mu ( X)$ is the Lebesgue measure of $X$  and $p_0$ is the smallest integer not less than $p/2$.
 \end{Example}
Note that Remark \ref{R2.3} shows that if we have sharp Nikol'skii inequalities, which is  the case  in (ii), (iii), and (iv), then
the corresponding Remez inequalities obtained above are also sharp.

 \subsection{Discretization inequality implies Remez inequality}
 We prove the following theorem here.
 \begin{Theorem}\label{T2.1} Let $f$ be a continuous periodic function on $\T^d$. Assume that there exists a set $X_m=\{\bx^j\}_{j=1}^m \subset \T^d$ such that for all functions $f_\by(\bx):= f(\bx-\by)$ we have the discretization inequality
\be\label{2.12}
\|f_\by\|_\infty \le D\max_{1\le j\le m}|f_\by(\bx^j)|.
\ee
Then for any $B$ with $|B|<1/m$ we have (\ref{2.1}) with $R=D$.
\end{Theorem}
\begin{proof} Consider the function
$$
g(\by) := \sum_{j=1}^m \chi_B(\bx^j-\by).
$$
At each point $\by$ either $g(\by)=0$ or $g(\by)\ge 1$. We prove that for $B$, $|B|<1/m$ there is a point $\by^*$ such that $g(\by^*)=0$. We prove this by contradiction. If such point $\by^*$ does not exist then $g(\by)\ge 1$ for all $\by\in\T^d$ and
$$
\int_{\T^d} gd\mu \ge 1.
$$
On the other hand
$$
\int_{\T^d} gd\mu = m|B|<1.
$$
The obtained contradiction proves the existence of $\by^*$ such that $g(\by^*)=0$. This implies in turn that for all $j$ we have $\chi_B(\bx^j -\by^*)=0$ or, in other words, $\bx^j -\by^* \in B^c:=\T^d\setminus B$. Next, by (\ref{2.12})
$$
\|f\|_\infty =\|f_{\by^*}\|_\infty \le D\max_{1\le j\le m}|f_{\by^*}(\bx^j)| = D\max_{1\le j\le m}|f(\bx^j-\by^*)| \le D \sup_{\bx\in B^c}|f(\bx)|.
$$
This completes the proof.
\end{proof}

\section{Hyperbolic cross polynomials}
\subsection{Remez inequality}
Let
$$
\Gm(N) :=\{\bk=(k_1,\dots,k_d): \prod_{j=1}^d \kb_j \le N, \quad \kb_j:=\max(1,|k_j|)\}
$$
be the hyperbolic cross and
$$
\Tr(N):=\{f(\bx), \bx=(x_1,\dots,x_d): f(\bx) = \sum_{\bk\in \Gm(N)}c_\bk e^{i(\bk,\bx)}\}.
$$
\begin{Theorem}\label{T3.1} There exist two positive constants $C_1(d)$ and $C_2(d)$ such that for any set $B\subset \T^d$ of normalized measure $|B|\le (C_2(d)N(\log N)^{d-1})^{-1}$ and for any
$f\in \Tr(N)$ we have
\be\label{3.1}
\|f\|_\infty \le C_1(d)(\log N)^{d-1} \sup_{\bu\in \T^d \setminus B} |f(\bu)|.
\ee
\end{Theorem}
\begin{proof}
Denote by $\V_N$ the de la Vall{\' e}e Poussin kernel for the hyperbolic cross $\Gm(N)$:
${\hat \V}_N(\bk)=1$ for $\bk\in \Gm(N)$ and ${\hat \V}_N(\bk)=0$ for $\bk\notin \Gm(2^dN)$. It is known (see, for instance, \cite{Tmon}, Chapter 1) that there exists a kernel $\V_N$ with the following properties:
\begin{equation}\label{3.2}
\|\V_N\|_1 \le C'(d)(\log N)^{d-1},\qquad \|\V_N\|_\infty \le C''(d) N(\log N)^{d-1}   .
\end{equation}
Then for any $f\in \Tr(N)$ we have $f=f\ast \V_n$, where $\ast$ means convolution.

Let $B$ be a set of small measure. We have for $f\in \Tr(N)$
$$
\|f\|_\infty = \|(2\pi)^{-d}\int_{\T^d} f(\bu)\V_N(\bx-\bu)d\bu\|_\infty
$$
$$
=\|(2\pi)^{-d}\int_{\T^d\setminus B} f(\bu)\V_N(\bx-\bu)d\bu + (2\pi)^{-d}\int_{B} f(\bu)\V_N(\bx-\bu)d\bu\|_\infty
$$
$$
\le  \max_{\bu\in \T^d\setminus B} |f(\bu)|\|\V_N\|_1 +  |B| \|f\|_\infty \|\V_N\|_\infty
$$
$$
\le C'(d)(\log N)^{d-1} \max_{\bu\in \T^d\setminus B} |f(\bu)| + C''(d) N(\log N)^{d-1}  |B| \|f\|_\infty.
$$
If $C''(d) N(\log N)^{d-1} |B|\le1/2$ then
$$
\|f\|_\infty \le 2C'(d)(\log N)^{d-1} \max_{\bu\in \T^d\setminus B} |f(\bu)|.
$$
This completes the proof with $C_1(d) := 2C'(d)$ and $C_2(d) := 2C''(d) $.
\end{proof}

Theorem \ref{T3.1} cannot be improved in a certain sense. The following statement holds for $d\ge 2$.

\begin{Proposition}\label{P3.1} The following statement is false: There exist  $\de>0$, $A$, $c$, and $C$ such that for any $f\in \Tr(N)$ and any set $B\subset \T^d$ of measure
$|B| \le (cN(\log N)^A)^{-1}$ the Remez-type inequality holds
\be\label{3.3}
\|f\|_\infty \le C(\log N)^{(d-1)(1-\de)} \sup_{\bu\in \T^d\setminus B} |f(\bu)|.
\ee
\end{Proposition}
\begin{proof} We use Proposition \ref{P2.1} with $p=\infty$. Our assumption (\ref{3.3}) gives (\ref{2.1}) with $b=(cN(\log N)^A)^{-1}$ and $R=C(\log N)^{(d-1)(1-\de)}$. Therefore, by Proposition \ref{P2.1} with $p=\infty$ (see also (\ref{2.7})) we get for all $f\in\Tr(N)$
\be\label{3.4}
\|f\|_\infty \le Rb^{-1/q}\|f\|_q,\quad 1\le q<\infty.
\ee
It is known (see \cite{Tmon} and Theorem \ref{T3.3} below) that it should be
\be\label{3.5}
Rb^{-1/q} \ge C(d,q) N^{1/q} (\log N)^{(d-1)(1-1/q)}.
\ee
Substituting our $b$ and $R$ expressed in terms of $N$ and choosing large enough $q$ and $N$, we obtain a contradiction in (\ref{3.5}).
\end{proof}

By (\ref{2.6}) Theorem \ref{T3.1} implies the following Remez-type inequality for all $0<p<\infty$.
\begin{Corollary}\label{C3.1} There exist two positive constants $C_1(d,p)$ and $C_2(d)$ (this constant is from Theorem \ref{T3.1}) such that for any set $B\subset \T^d$ of normalized measure $|B|\le (2C_2(d)N(\log N)^{d-1})^{-1}$ and for any
$f\in \Tr(N)$ we have
\be\label{3.6p}
\|f\|_p \le C_1(d,p)(\log N)^{d-1} \|f\|_{L_p( \T^d \setminus B)}.
\ee
\end{Corollary}

  Proposition \ref{P2.2}, Remark \ref{R2.1}, and the Nikol'skii inequalities in Theorem \ref{T3.6}, allow us to improve the above Corollary \ref{C3.1}.

\begin{Theorem}\label{T3.1p} For $0<q<\infty$  there exist two positive constants $C_1(d,q)$ and $C_2(d,q)$ such that for any set $B\subset \T^d$ of normalized measure $|B|\le (C_2(d,q)N)^{-1}$ and for any
$f\in \Tr(N)$ we have
\be\label{3.7p}
\|f\|_q \le C_1(d,q)  \|f\|_{L_q( \T^d \setminus B)}.
\ee
\end{Theorem}

Theorems \ref{T3.1p} and Theorem \ref{T3.6} imply the following combination of Nikol'skii-type and Remez-type inequalities.

\begin{Theorem}\label{TNR1} For $0<q\le p<\infty$  there exist two positive constants $C_1=C_1(d,p,q)$ and $C_2=C_2(d,p,q)$ such that for any set $B\subset \T^d$ of normalized measure $|B|\le (C_2N)^{-1}$ and for any
$f\in \Tr(N)$ we have
\be\label{3.7pq}
\|f\|_p \le C_1 N^\bt  \|f\|_{L_q( \T^d \setminus B)},\quad \bt:=1/q-1/p.
\ee
\end{Theorem}

\subsection{Improved Remez inequality in case $d=2$}
Proposition \ref{P3.1} shows that we cannot substantially improve on the additional factor $(\log N)^{d-1}$ in (\ref{3.1}) of Theorem \ref{T3.1}. In this subsection we will improve the bound $b$ on the measure of a set $B$. Our technique is based on the Riesz products. It works in the case $d=2$. We introduce some notations. Let $\bs=(s_1,\dots,s_d)$ be a vector with nonnegative integer coordinates ($\bs \in \Z^d_+$) and
$$
\rho(\bs):= \{\bk=(k_1,\dots,k_d)\in \Z^d_+ : [2^{s_j-1}]\le |k_j|<2^{s_j},\quad j=1,\dots,d\}
$$
where $[a]$ denotes the integer part of a number $a$.  Denote for a natural number $n$
$$
Q_n := \cup_{\|\bs\|_1\le n}\rho(\bs); \qquad \Delta Q_n := Q_n\setminus Q_{n-1} = \cup_{\|\bs\|_1=n}\rho(\bs)
$$
with $\|\bs\|_1 = s_1+\dots+s_d$ for $\bs\in \Z^d_+$. We call a set $\Delta Q_n$ {\it hyperbolic layer}. For a set $\L \subset \Z^d$ denote
$$
\Tr(\L) := \{f\in L_1:  \hat f(\bk)=0, \bk\in \mathbb Z^d\setminus \L\} .
$$

For any two integers $a\ge 1$ and $0\le b<a$, we shall denote by $AP(a,b)$ the arithmetical progression of the form $al+b$, $l=0,1,\dots$. Set
$$
H_n(a,b) := \{\bs=(s_1,s_2): \bs\in \Z_+^2,\quad \|\bs\|_1=n, \quad s_1,s_2\ge a,\quad s_1\in AP(a,b)\}.
$$
 Define
$$
 \rho'(\bs) := \{\bm=(m_1,m_2): [2^{s_i-2}]\le |m_i|< 2^{s_i}, i=1,2\}.
$$

Let us define the polynomials $\mathcal A_{\mathbf s} (\mathbf x)$
for $\mathbf s =  (s_1 ,\dots,s_d)\in\N^d_0$
$$
\mathcal A_{\mathbf s} (\mathbf x) :=\prod_{j=1}^d\mathcal A_{s_j}(x_j),
$$
with $\mathcal A_{s_j}(x_j)$ defined as follows:
$$
\mathcal A_0 (x) := 1, \quad \mathcal A_1 (x) := \mathcal V_1 (x) - 1, \quad
\mathcal A_s (x) := \mathcal V_{2^{s-1}} (x) -\mathcal V_{2^{s-2}} (x),
\quad s\ge 2,
$$
where $\mathcal V_m$ are the de la Vall\'ee Poussin kernels.
Then for $d=2$
$$
\A_\bs \in \Tr(\rho'(\bs)).
$$

For a subspace $Y$ in $L_2(\T^d)$ we denote by $Y^\perp$ its orthogonal complement.
We need the following lemma on the Riesz product, which is Lemma 2.1 from \cite{TE3}.

\begin{Lemma}\label{L2.1v} Take any trigonometric polynomials $t_\bs\in \Tr(\rho'(\bs))$ and form the function
$$
\Phi(\bx) := \prod_{\bs\in H_n(a,b)}(1+t_\bs).
$$
Then for any $a\ge 6$ and any $0\le b<a$ this function admits the representation
$$
\Phi(\bx) = 1+ \sum_{\bs\in H_n(a,b)} t_\bs(\bx) +g(\bx)
$$
with $g\in \Tr(Q_{n+a-6})^\perp$.
\end{Lemma}

We remind that we restrict ourselves to $d=2$. Denote
$$
t_\bs := \A_\bs/M,\qquad M:= \max_{\bs\in H_n(a,b)} \|\A_s\|_\infty\asymp 2^n.
$$
Consider the Riesz product
$$
\Phi := \prod_{\bs\in H_n(a,b)} \left(1+\frac{it_\bs}{\sqrt{N}}\right),\quad N:= |H_n(a,b)|.
$$
Then it is easy to derive from the inequality $\left|1+\frac{it_\bs}{\sqrt{N}}\right| \le \left(1+\frac{1}{N}\right)^{1/2}$ that (see Remark 2.1 from \cite{TE4})
$$
|\Phi| \le C.
$$
Moreover, by Lemma \ref{L2.1v} we have
$$
\Phi = 1+\frac{i}{\sqrt{N}}\sum_{\bs\in H_n(a,b)} t_\bs + w,\quad w \in \Tr(Q_{n+a-6})^\perp.
$$
Thus,
\be\label{3.6}
\| \sum_{\bs \in H_n(a,b)} t_\bs + N^{1/2}\text{Im}(w)\|_\infty \le CN^{1/2}.
\ee
We now bound $\|w\|_1$. We introduce some more notations. Denote
$$
H_n^k := \{(\bs^1,\dots,\bs^k): \bs^j\in H_n(a,b), j=1,\dots,k, \quad \text{are distinct} \}
$$
$$
h_n := AP(a,b)\cap [a,n-a].
$$
We have
$$
w = \sum_{k=2}^N \left(\frac{i}{N^{1/2}M}\right)^k\sum_{(\bs^1,\dots,\bs^k)\in H_n^k} \prod_{j=1}^k \A_{\bs^j}
$$
$$
= \sum_{k=2}^N \left(\frac{i}{N^{1/2}M}\right)^k\sum_{s^1_1 \in h_n} \sum_{s^2_1\in h_n: s^2_1 <s^1_1} \dots \sum_{s^k_1\in h_n: s^k_1 <s^{k-1}_1} \prod_{j=1}^k \A_{\bs^j}.
$$
Therefore,
\be\label{3.7}
\|w\|_1 \le  \sum_{k=2}^N \left(\frac{1}{N^{1/2}M}\right)^k\sum_{s^1_1 \in h_n} \sum_{s^2_1\in h_n: s^2_1 <s^1_1} \dots \sum_{s^k_1\in h_n: s^k_1 <s^{k-1}_1} \|\prod_{j=1}^k \A_{\bs^j}\|_1.
\ee
Next,
$$
\|\prod_{j=1}^k \A_{\bs^j}\|_1 \le \|\A_{s^1_1}(x_1)\|_1 \prod_{j=2}^k \|\A_{s^j_1}(x_1)\|_\infty
\|\A_{s^k_2}(x_2)\|_1 \prod_{j=1}^{k-1} \|\A_{s^j_2}(x_2)\|_\infty
$$
\be\label{3.8}
\le C2^{s^2_1+\cdots+s^k_1+n-s^1_1 +\cdots+ n-s^{k-1}_1}.
\ee
Inequalities (\ref{3.7}) and (\ref{3.8}) imply
$$
\|w\|_1 \le C\sum_{k=2}^N \left(\frac{1}{N^{1/2}M}\right)^k N2^{n(k-1)} \ll 2^{-n}.
$$
Therefore,
\be\label{3.9}
\|  M N^{1/2} \text{Im}(w)\|_1 \ll N^{1/2}.
\ee
Bounds (\ref{3.9}) and (\ref{3.6}) with $a=6$ imply that there exist a function $t\in \Tr(Q_n)^\perp$ such that
\be\label{3.10}
\| \sum_{\bs\in H_n(a,b)} \A_\bs -t\|_1 \ll n,
\ee
and
\be\label{3.11}
\| \sum_{\bs\in H_n(a,b)} \A_\bs -t\|_\infty \ll n^{1/2}2^n.
\ee
Consider
$$
\Delta \V_n := \sum_{\bs: n\le \|\bs\|_1\le n+2} \A_\bs.
$$
Note that for any $f\in \Tr(\Delta Q_n)$ we have $f\ast \Delta \V_n =f$.
The above inequalities (\ref{3.10}) and (\ref{3.11}) imply the following assertion.
\begin{Lemma}\label{L3.1} There exists $T\in \Tr(Q_n)^\perp$ such that
$$
\|\Delta \V_n - T\|_1 \ll n,\quad \|\Delta \V_n - T\|_\infty \ll n^{1/2} 2^n.
$$
\end{Lemma}
In the same way as Theorem \ref{T3.1} was derived from inequalities (\ref{3.2}) the following theorem can be derived from Lemma \ref{L3.1}.
\begin{Theorem}\label{T3.2} Let $d=2$. There exist two positive constants $C_1$ and $C_2$ such that for any set $B\subset \T^2$ of normalized measure $|B|\le (C_22^nn^{1/2})^{-1}$ and for any
$f\in \Tr(\Delta Q_n)$ we have
\be\label{3.12}
\|f\|_\infty \le C_1n \sup_{\bu\in \T^2 \setminus B} |f(\bu)|.
\ee
\end{Theorem}

\subsection{The Nikol'skii inequalities}
The following two theorems are from \cite{Tmon}, Ch.1, Section 2.
\begin{Theorem}\label{T3.3} Suppose that $1 \le q < \infty $. Then
$$
\sup_{f\in \Tr(N)}\|f\|_{\infty}
/ \|f\|_q\asymp N^{1/q}(\log N)^{(d-1)(1-1/q)}.
$$
\end{Theorem}

\begin{Theorem}\label{T3.4} Suppose that $1 \le q \le p < \infty $. Then
$$
\sup_{f\in \Tr(N)}\|t \|_p /
\|f\|_q\asymp N^{1/q-1/p}.
$$
\end{Theorem}
In this subsection we extend the above two theorems to the range of parameters $0<q<p\le \infty$. We begin with the case $p=\infty$.

\begin{Theorem}\label{T3.5} Suppose that $0< q < \infty $. Then
$$
\sup_{f\in \Tr(N)}\|f\|_{\infty}
/ \|f\|_q\asymp N^{1/q}(\log N)^{(d-1)(1-1/q)_+}.
$$
\end{Theorem}
\begin{proof} We prove the upper bound in the case $0<q<1$. The corresponding lower bounds in this case follow from the univariate case. We derive the required inequality from Theorem \ref{T3.3} with $q=1$. Let $f\in \Tr(N)$. Then
$$
\|f\|_1 = \| |f|^{1-q} |f|^q\|_1 \le \||f|^{1-q}\|_\infty \||f|^q\|_1 = \|f\|_\infty^{1-q} \|f\|_q^q.
$$
Applying Theorem \ref{T3.3} with $q=1$ we continue
$$
\le C(d)N \|f\|_1 \|f\|_\infty^{-q} \|f\|_q^q.
$$
This implies
\be\label{3.13}
\|f\|_\infty^q \le C(d)N\|f\|_q^q \quad \text{and}\quad \|f\|_\infty \le (C(d)N)^{1/q}\|f\|_q.
\ee
which completes the proof.
\end{proof}

\begin{Theorem}\label{T3.6} Suppose that $0 < q < p < \infty $. Then
$$
\sup_{f\in \Tr(N)}\|t \|_p /
\|f\|_q\asymp N^{1/q-1/p}.
$$
\end{Theorem}
\begin{proof} We prove the upper bound in the case $0<q<1$. 
 We have
$$
\|f\|_p = \||f|^{1-q/p}|f|^{q/p}\|_p \le \|f\|_\infty^{1-q/p} \|f\|_q^{q/p}.
$$
Using Theorem \ref{T3.5} we continue
$$
\le (C(d)N)^{(1-q/p)/q}\|f\|_q^{1-q/p}\|f\|_q^{q/p} = (C(d)N)^{\bt}\|f\|_q,\quad \bt:=1/q-1/p.
$$

The sharpness of Nikolskii's  inequality, i.e., the part $"\gg"$, follows from
the one-dimensional  Jackson kernel example:
$$T(x)=\left(\frac{\sin \frac{nt}2}{n \sin\frac{t}2}\right)^{2r}, \qquad r\in \N.$$
See \cite[\S 4.9]{timan} for $1\le p \le \infty$; in the case of $0<p<1$ it is enough to take $r$ large enough ($r>\frac1{2p}$).

\end{proof}

\subsection{Discretization}
An operator $T_n$ with the following properties was constructed in \cite{T93}. The operator $T_n$ has the form
$$
T_n(f) = \sum_{j=1}^m f(\bx^j) \psi_j(\bx),\quad m\le c(d)2^n n^{d-1},\quad \psi_j \in \Tr(Q_{n+d})
$$
and
\be\label{3.21}
T_n(f) =f,\quad f\in \Tr(Q_n),
\ee
\be\label{3.22}
\|T_n\|_{L_\infty\to L_\infty} \asymp n^{d-1}.
\ee
Properties (\ref{3.21}) and (\ref{3.22}) imply that all $f\in\Tr(Q_n)$ satisfy the discretization inequality (see \cite{KT} and \cite{DTU}, subsection 2.5)
\be\label{3.23}
\|f\|_\infty \le C(d)n^{d-1} \max_{1\le j\le m} |f(\bx^j)|.
\ee
Note that Theorem \ref{T2.1} and the discretization inequality (\ref{3.23}) give other proof of Theorem \ref{T3.1}.
Theorem \ref{T2.1} and Proposition \ref{P3.1} imply the following assertion.

\begin{Proposition}\label{P3.2} The following statement is false: There exist  $\de>0$, $A$, $c$, and $C$ such that there exists a set $X_m=\{\bx^j\}_{j=1}^m$ with $m\le c2^nn^A$, which provides the discretization inequality for $\Tr(Q_n)$:
$$
\|f\|_\infty \le Cn^{(d-1)(1-\delta)} \max_{1\le j\le m} |f(\bx^j)|,\quad f\in\Tr(Q_n).
$$
\end{Proposition}
Thus, an extra factor $n^{d-1}$ in the discretization inequality for $\Tr(Q_n)$ cannot be substantially improved, if we limit ourselves to the number of $m\ll 2^nn^A$ points. It is proved in \cite{KT} (see \cite{DTU}, subsection 2.5, for a discussion) that in the case $d=2$ in order to drop the extra factor $n$ in (\ref{3.23}) we need to use at least $2^{n(1+c_0)}$, $c_0>0$, points. It is clear that the necessary condition for the discretization inequality (\ref{3.23}) to hold with some extra factor is $m\ge |Q_n| =\dim\Tr(Q_n) \asymp 2^n n^{d-1}$.
Therefore, the way from discretization inequality to the Remez inequality, provided by Theorem \ref{T2.1}, cannot give a better bound than $b(Q_n)\asymp (2^n n^{d-1})^{-1}$.
However, the direct proof of the Remez inequality in Theorem \ref{T3.2} gives for $d=2$ a better bound $b(\Delta Q_n) \asymp (2^n n^{1/2})^{-1}$.

{\bf Acknowledgements.}
This research was carried out  when the authors visited the Centre de Recerca Matem\`{a}tica in Barcelona and the Hausdorff Research Institute for Mathematics in Bonn.
The first named author was partially supported by the Clay Mathematical Institute grant.
The second  named author was partially supported by
 MTM 2014-59174-P, 2014 SGR 289, and the Hausdorff Research Institute.

\end{document}